\documentclass[11pt, a4paper]{article}

\pdfoutput=1
\usepackage{graphicx}
\graphicspath{{./Figures/}}
\usepackage{subfigure}

\usepackage{amssymb,latexsym}

\usepackage[nocompress]{cite}

\usepackage{setspace}

\usepackage{amsmath}

\usepackage{mathrsfs,amsfonts}

\usepackage{amsthm,amsxtra}

\usepackage[final]{showkeys}

\usepackage{stmaryrd}

\usepackage{algorithm}
\usepackage{algorithmicx}
\usepackage{algpseudocode}
\usepackage{mathtools}

\usepackage{yhmath}

\usepackage[debug,pdftex,colorlinks=true, 
linkcolor=blue, bookmarksopen=false,
plainpages=false,pdfpagelabels]{hyperref}

\usepackage[openbib]{currvita}

\usepackage{fancyhdr}

\newtheorem{theorem}{Theorem}[section]
\newtheorem{lemma}[theorem]{Lemma}

\newtheorem{remark}[theorem]{Remark} 
\newtheorem{corollary}[theorem]{Corollary}

\newcommand{\eps}{\varepsilon}

\newcommand{\bbC}{\mathbb C} 
 \newcommand{\bbN}{\mathbb N}
\newcommand{\bbP}{\mathbb P} 
\newcommand{\bbR}{\mathbb R} 
\newcommand{\bbZ}{\mathbb Z} 

\newcommand{\bzero}{{\mathbf 0}}

\newcommand{\be}{\mathbf e}

\newcommand{\bu}{\mathbf u} \newcommand{\bv}{\mathbf v} 
  
 \newcommand{\bz}{\mathbf z} 
\newcommand{\bA}{\mathbf A} \newcommand{\bB}{\mathbf B}
  
\newcommand{\bE}{\mathbf E} \newcommand{\bF}{\mathbf F}
 \newcommand{\bH}{\mathbf H}
\newcommand{\bI}{\mathbf I} 
 
\newcommand{\bM}{\mathbf M} 
  
\newcommand{\bQ}{\mathbf Q} \newcommand{\bR}{\mathbf R}

\newcommand{\bW}{\mathbf W} \newcommand{\bX}{\mathbf X} 
 \newcommand{\bZ}{\mathbf Z}

 \newcommand{\cD}{\mathcal D} 
 
 \newcommand{\cH}{\mathcal H}

\newcommand{\cO}{\mathcal O}  
\newcommand{\cQ}{\mathcal Q} \newcommand{\cR}{\mathcal R}
\newcommand{\cS}{\mathcal S} 
 \newcommand{\cV}{\mathcal V}

\newcommand{\aver}[1]{\langle {#1} \rangle}

\newenvironment{keywords}
{\noindent{\bf Key words.}\small}{\par\vspace{1ex}}

\makeatletter
\newcommand{\chapterauthor}[1]{%
	{\parindent0pt\vspace*{-25pt}%
		\linespread{1.1}\large\scshape#1%
		\par\nobreak\vspace*{35pt}}
	\@afterheading%
}
\makeatother

\usepackage{xcolor}

\title{When the noncommutative AM-GM inequality holds}
\author{Yimin Zhong\thanks{Department of Mathematics and Statistics, Auburn University, Auburn, AL 36830;
yimin.zhong@auburn.edu}}

\date{}

\begin{document}
\maketitle

\begin{abstract}
In this work, we prove that the noncommutative AM-GM inequality holds if $n\ge 2\lceil m/2 \rceil^2$. The motivation comes from counterexamples constructed in~\cite{de2020random}. The proof constructs a vertex measure based on the Chebyshev nodes on the Boolean cube to extract the distinct indices. The main difficulty is that the measure is not positive on non-integer nodes. The key technique comes from~\cite{grigoriev2001complexity} and eventually transforms the problem into a quadrature estimate. 
\end{abstract}

\section{Introduction}
In 2012, Recht and R\'e~\cite{recht2012toward} conjectured a noncommutative arithmetic-geometric mean (AM-GM) inequality: given $n$ positive semidefinite (PSD) matrices $\mathbf{A}_1,\dots, \mathbf{A}_n$, such that $\sum_{i=1}^n \mathbf{A}_i \preceq n \mathbf{I}$, then for any $1\le m\le n$, 
\begin{equation}\label{eq: rr}
  -\frac{n!}{(n-m)!} \mathbf{I} \preceq \sum_{\substack{1\le j_1,\dots, j_m \le n\\ j_1,\dots, j_m\;\text{distinct}}} \prod_{l=1}^m \mathbf{A}_{j_l} \preceq \frac{n!}{(n-m)!} \mathbf{I}.
\end{equation}
The inequality is closely related to stochastic gradient descent (SGD), suggesting that without-replacement sampling should help the training converge faster than with-replacement sampling, where $n$ is the number of samples and $m$ is the number of disjoint batches. The inequality was subsequently proved for $m\le 3$ with all $n\ge m$~\cite{recht2012toward,zhang2018note,lai2020recht}. In 2020, Lai and Lim~\cite{lai2020recht} disproved it through noncommutative Positivstellensatz theorem~\cite{helton2002positive} for $(m, n) =(5,5)$, and meanwhile De Sa~\cite{de2020random} provided a systematic way to construct counterexamples for both sides of~\eqref{eq: rr}. Soon after, Yun, Sra, and Jadbabaie~\cite{yun2021open} conjectured that~\eqref{eq: rr} should hold for $m=n$ on nearly identity matrices, implying that without-replacement sampling should still be superior with tiny learning rates for SGD, and~\cite{peng2026resolution} recently confirmed this.

Inspired by De Sa's counterexamples~\cite{de2020random}, this note studies whether the AM-GM inequality~\eqref{eq: rr} holds if $n$ is sufficiently large compared to $m^{2}$. This is exactly where those counterexamples begin to disappear.

\subsection{Motivation}\label{sec: mot}
Our motivation comes from a slight modification of De Sa's counterexamples~\cite{de2020random}. It reveals the mechanism by which the inequality is violated.

Let $n\ge 3$ and $d = n-1$. Denote $\be_1,\dots, \be_n$ the standard orthonormal basis in $\bbR^n$, and define vectors
\begin{equation}\label{eq: uz}
    \bu = \frac{1}{\sqrt{n}} \mathbf{1}, \qquad \bz_i = \sqrt{\frac{n}{d}} (\mathbf{I} - \bu\bu^{\top}) \be_i =  \sqrt{\frac{n}{d}} \left(\be_i - \frac{\bu}{\sqrt{n}}\right).
\end{equation}
Then $\|\bu\|_2 = \|\bz_i\|_2 = 1$, $\bu\perp \bz_i$, and 
\begin{equation}\label{eq: zz}
    \bz_i^{\top} \bz_j = -\frac{1}{d}{\quad(i\ne j)}, \qquad \sum_{i=1}^n \bz_i = 0.
\end{equation}
We construct the PSD matrices in $\bA_i\in\bbR^{n\times n}$ as follows.
\begin{equation}\label{eq: ai}
\begin{aligned}
    \bA_i &= \alpha \mathbf{I} + (1 -\alpha) \bu \bu^{\top} + \beta (\bu \bz_i^{\top} + \bz_i\bu^{\top}) + (\beta^2 - \alpha) \bz_i \bz_i^{\top} \\
    &= \alpha (\mathbf{I} - \bu\bu^{\top} - \bz_i\bz_i^{\top}) + \begin{pmatrix}
        \bu & \bz_i 
    \end{pmatrix} \begin{pmatrix}
        1\\ \beta 
    \end{pmatrix} \begin{pmatrix}
        1 & \beta 
    \end{pmatrix} \begin{pmatrix}
        \bu^{\top} \\ \bz_i^{\top} 
    \end{pmatrix},
\end{aligned}
\end{equation}
where $\alpha = \frac{d - \beta^2}{d-1}$ and $0\le \beta^2 \le d$. Then each $\bA_i$ is PSD by~\eqref{eq: ai}. Using~\eqref{eq: uz} and~\eqref{eq: zz}, we obtain
\begin{equation*}
\begin{aligned}
    \sum_{i=1}^n \bA_i &=  n {\left(\alpha\mathbf{I}+(1-\alpha)\bu\bu^{\top}\right)} + (\beta^2 -\alpha) \sum_{i=1}^n \bz_i\bz_i^{\top} \\
    &=n {\left(\alpha\mathbf{I}+(1-\alpha)\bu\bu^{\top}\right)}+ (\beta^2 -\alpha)  \frac{n}{d} (\mathbf{I} - \bu\bu^{\top})  \sum_{i=1}^n \be_i \be_i^{\top} (\mathbf{I} - \bu\bu^{\top})  \\
    &= n\mathbf{I}.
\end{aligned}
\end{equation*}
Let
\begin{equation}\label{eq: W}
    \bW_{m,n} := \frac{(n-m)!}{n!}\sum_{\substack{1\le j_1,\dots, j_m \le n\\ j_1,\dots, j_m\;\text{distinct}}} \prod_{l=1}^m \mathbf{A}_{j_l}.
\end{equation}
 Then $\bW_{m, n}$ commutes with all permutation matrices since $P_{\sigma}\bA_i P_{\sigma}^{\top} = \bA_{\sigma(i)}$ for any $\sigma\in{\rm Sym}(n)$ and the summation runs through all permutations. Thus, $\mathbf{1}$ is an eigenvector of $\bW_{m, n}$. 
 
 Let $\ell\in\{1,2\}$ and $m\ge 2\ell + 2$. We choose the constants in $\bW_{m,  n}$ as follows:
 \begin{equation}\label{eq: choice}
     \theta = \frac{\ell \pi}{m-1},\qquad \alpha = \frac{n}{n-1}\cos^2\theta, \qquad \beta = \sqrt{d - (d-1)\alpha}. 
 \end{equation}
Then we can show that the eigenvalue for $\mathbf{1}$ is
\begin{equation}\label{eq: eig}
    \lambda = (-1)^{\ell}\left(\frac{n}{n-1} \cos \theta \right)^{m-1},\qquad \ell\in\{1, 2\}.
\end{equation}
We postpone the proof of~\eqref{eq: eig} to Appendix~\ref{sec: p eig}, which follows a similar route as~\cite{de2020random}. 
To make~\eqref{eq: rr} hold on both sides, we need $|\lambda|\le 1$ for both $\ell\in\{1, 2\}$, which is equivalent to 
\begin{equation*}
    n \ge \frac{1}{1-\cos\theta} = \frac{2(m-1)^2}{\ell^2 \pi^2} + \cO(1).
\end{equation*}
The threshold $n = \Omega(m^2)$ is not surprising. Consider sampling $m$ matrices from $\{\bA_1, \dots, \bA_n\}$, the probability of no repetition is
\begin{equation*}
    \prod_{i=0}^{m-1} \left(1 - \frac{i}{n}\right)
    {\approx \prod_{i=0}^{m-1} \exp\left(-\frac{i}{n}\right) = \exp\left(-\frac{m(m-1)}{2n}\right)}.
\end{equation*}
Thus $n=\Omega(m^2)$ is the dominance threshold for sampling with and without replacement. This inspires the question:

\textit{Is there a universal constant $C_{\ast} > 0$ such that if $n \ge C_{\ast} m^2$, the inequality~\eqref{eq: rr} holds on both sides?}

Note that this adds no further conditions on the PSD matrices themselves, distinguishing it from~\cite{yun2021open, peng2026resolution}, which requires nearly identity PSD matrices. 

\subsection{Main results}
Our main results are the following two theorems. 

\begin{theorem}\label{thm: low}
    If $n\ge 4m^{2}$, the AM-GM inequality~\eqref{eq: rr} holds on the lower bound and
    \begin{equation*}
        \bW_{m,n} \succeq -\left(1 - \frac{1}{2m}\right)\bI.
    \end{equation*}
\end{theorem}

\begin{theorem}\label{thm: both}
    If $n \ge 2 \lceil m/2 \rceil^2$, the AM-GM inequality~\eqref{eq: rr} holds on both bounds.
\end{theorem}

Theorem~\ref{thm: low} and Theorem~\ref{thm: both} answer the question positively. The constant satisfies $\frac{1}{2} \ge C_{\ast} \ge \frac{2}{\pi^2}$ as $m\to \infty$, but the optimal value remains unknown (likely $4/\pi^2$). The lower bound in~\eqref {eq: rr} is easier to prove because it is generally \textit{not} sharp, as numerical evidence suggests in~\cite{lai2020recht}. Hence, it leaves enough room for relaxation techniques. However, proving the upper bound appears harder, since the equality can hold sharply at $\bA_i= \mathbf{I}$.  

The rest of this note is organized as follows. The proofs of Theorem~\ref{thm: low} and Theorem~\ref{thm: both} are in Section~\ref{sec: low} and Section~\ref{sec: high}, respectively. 
We emphasize that the constants in the above theorems are not fully optimized. The concluding remarks are in Section~\ref{sec: con}.

\subsection{AI disclosure}
The author did not use any AI tools for Theorem~\ref{thm: low}. After we gave GPT-6 Astra the proof of Theorem~\ref{thm: low}, it returned the key reference~\cite{grigoriev2001complexity} (see Theorem~\ref{thm: pos}), and using its conclusion, we immediately obtained $m=\cO(n^{1/3})$ in Theorem~\ref{thm: m3}. 
The proof of Theorem~\ref{thm: both} was found by GPT-6 Astra gradually. It first improved our $m = \cO(n^{1/3})$ to an intermediate regime $m=\cO(n^{3/8})$ by a coarse estimate on the negative contributions. It soon found a complicated approach that bookkept various contributions to $m = \cO(n^{1/2}\log\log n/\log n)$ but could not lift the final logarithmic factor. After we asked it to exploit the specialness of parity Chebyshev nodes, it realized we had neglected an important polynomial property of the harmonic decomposition. Then it found an elegant route to Theorem~\ref{thm: both}. 

We should notice that GPT-6 Astra's Theorem~\ref{thm: both} covers part of our Theorem~\ref{thm: low} and Theorem~\ref{thm: m3}. However, we believe the proofs of Theorem~\ref{thm: low} and Theorem~\ref{thm: m3} are still useful, reveal the underlying difficulties, and help understanding. The author is responsible for all of the presented work.

\section{Proof on lower bound}\label{sec: low}
Let $\bbP_{m}$ be real-valued degree-$m$ polynomials. For a finite signed measure $\nu$, let $\|\nu\|_{\rm TV}$ be the total variation norm of $\nu$. Denote the normalized PSD matrices $\bE_i := \frac{1}{n}\bA_i$, then $\bM:= \sum_{i=1}^n \bE_i \preceq \mathbf{I}$. 
Define the base matrix-valued polynomial on the Boolean hypercube:
\begin{equation*}
    \bB(\bz) = \sum_{i=1}^n z_i \bE_i,  \qquad \bz = (z_1,\dots, z_n)\in \{0, 1\}^n.
\end{equation*}
Then, $\bB^p(\bz)\preceq \mathbf{I}$ for any $p\ge 0$.  Let  $F$ be a scalar degree-$m$ polynomial with leading coefficient as $\mathfrak{C} \neq 0$, then 
\begin{equation}\label{eq: W coefficient}
    \bW_{m,n}=\frac{1}{\mathfrak{C}}\frac{n^m(n-m)!}{n!}
       \sum_{\substack{S\subset[n]\\ |S|=m}}[z_S]F(\bB(\bz)),
       \qquad z_S:=\prod_{i\in S}z_i.
\end{equation}
Any monomial on the Boolean hypercube is square-free, using $z^2=z$ for $z\in\{0, 1\}$. We always assume a polynomial has been reduced to the square-free form. The notation $[z_{i_1}\cdots z_{i_r}] f$ means the coefficient of $z_{i_1}\cdots z_{i_r}$ in $f$. 
To extract these coefficients, we need to design a finite signed measure $\mu_{m,n}$ that suppresses other terms along with the normalization constant $\mathfrak{C}$:
\begin{equation}\label{eq: w nm int}
    \bW_{m,n} = \int_{\{0, 1\}^n} F(\bB(\bz)) d\mu_{m,n}(\bz).
\end{equation}
If the measure $\mu_{m, n}(\{\bzero\}) = p \ge 0$, then 
\begin{equation*}
\begin{aligned}
    \bW_{m,n} &= \int_{\{0, 1\}^n}  F(\bB(\bz)) d\mu_{m,n}(\bz) \\&=  F(0) p\mathbf{I}  +  \int_{\{0, 1\}^n - \{\bzero\}} F(\bB(\bz)) d\mu_{m,n}(\bz) \\
    &\succeq F(0) p \mathbf{I}  - \|F\|_{\infty}  (\|\mu_{m, n}\|_{\rm TV} - p) \mathbf{I} \\
    &=  \left( (F(0) + \|F\|_{\infty})p - \|F\|_{\infty} \|\mu_{m, n}\|_{\rm TV} \right) \mathbf{I}.
\end{aligned}
\end{equation*}
Therefore, we obtain a relaxation problem: finding a degree-$m$ polynomial $F$ and a finite signed measure $\mu_{m,n}$ such that~\eqref{eq: w nm int} holds and 
\begin{equation}\label{eq: cond}
     (F(0) + \|F\|_{\infty}) p - \|F\|_{\infty} \|\mu_{m, n}\|_{\rm TV} \ge -1.
\end{equation}
{For fixed interpolation nodes, the moment constraints for $\mathfrak{C}\mu_{m,n}$ depend only on $m,n$, not on the lower coefficients of $F$.}
The condition~\eqref{eq: cond} rules out the simple choice $F(t) = t^m$, which wastes the contribution from $F(0)$. A better candidate choice needs the polynomial satisfies $F(0) = \|F\|_{\infty}$ and ${|\mathfrak{C}|^{-1}\|F\|_{\infty}}$ as small as possible. This leaves us with the Chebyshev polynomial $T_m(1-2t)$ as a first choice for $F$, the transform $t\mapsto 1-2t$ maps $[0, 1]$ to the standard interval $[-1,1]$ for Chebyshev polynomials. The leading coefficient has sign $(-1)^m$. The remaining task becomes finding $\mu_{m,n}$ that satisfies the linear constraints~\eqref{eq: w nm int}, the inequality constraint $\mu_{m, n}(\{\bzero\}) > 0$, and minimizes the cost $\|\mu_{m,n}\|_{\rm TV} - 2 \mu_{m, n}(\{\bzero\}) $. This is a linear programming problem, but we do not need to solve it accurately. A simple relaxation on minimization of $\|\mu_{m,n}\|_{\rm TV}$ already gives a clean proof.

\subsection{Vertex measure}
For each $0\le s\le n$, let the vertex measure
\begin{equation*}
    \nu_{s} := \frac{1}{\binom{n}{s}} \sum_{\substack{S\subset [n]\\|S| = s}}\,\prod_{i\in S} \delta_{z_i=1} \otimes \prod_{i\notin S} \delta_{z_i=0}, 
\end{equation*}
which forms a uniform probability measure on the index sets $S\subset [n]$ with $|S|=s$. For any distinct indices $j_1,\dots, j_r$, we have 
\begin{equation*}
    \int_{\{0, 1\}^n} z_{j_1}\cdots z_{j_r} d\nu_s(\bz) = \frac{\binom{n-r}{s-r} }{\binom{n}{s}} = \frac{s(s-1)\cdots (s-r+1)}{n(n-1)\cdots (n-r+1)}.
\end{equation*}
The numerator $\binom{n-r}{s-r}$ counts the ways of selecting $s-r$ indices from the remaining $n-r$ indices to form a unique set of size $s$. In particular, if $r=0$, the integral yields one for all $s$. For convenience, we use the standard notation for later analysis.
\begin{equation*}
    (a)_b := a(a-1)\cdots (a-b+1),\qquad (a)_0 := 1,\qquad a\in\bbR, \; b\in\bbN.
\end{equation*}

Suppose the Chebyshev polynomial $   F(\bB(\bz))=T_m(\bI-2\bB(\bz)) $ expands to
\begin{equation}\label{eq: chby}
    F(\bB(\bz))=T_m(\bI-2\bB(\bz))
      =\sum_{r=0}^m\ \sum_{\substack{S\subset[n]\\|S|=r}}
         z_S\mathbf{C}_S,\qquad
    \mathbf{C}_S:=[z_S]T_m(\bI-2\bB(\bz)).
\end{equation}
For $m\ge1$ the scalar leading coefficient is
$\mathfrak{C}=(-1)^m2^{2m-1}$, and for $S=\{i_1,\ldots,i_m\}$,
\[
 \mathbf{C}_S=(-1)^m2^{2m-1}
       \sum_{\sigma\in\mathrm{Sym}(m)}
       \bE_{i_{\sigma(1)}}\cdots\bE_{i_{\sigma(m)}}.
\]
The integral with $\nu_s$ is
\begin{equation}\label{eq: fb}
 \int_{\{0,1\}^n}F(\bB(\bz))\,d\nu_s(\bz)
   =\sum_{r=0}^m\frac{(s)_r}{(n)_r}
          \sum_{\substack{S\subset[n]\\|S|=r}}\mathbf{C}_S.
\end{equation}
There are $n+1$ vertex measures $\nu_{s}$, $s=0,\cdots, n$. The constraints~\eqref{eq: w nm int} involve $m+1$ linear equations, which means we need to find distinct integers $s_0,\dots, s_m\in \{0, 1, \dots, n\}$ and corresponding coefficients $a_0,\dots, a_m$ such that 
\begin{equation}\label{eq: mu cons}
    \mu_{m,n} = \sum_{j=0}^m a_j \nu_{s_j},\qquad     \sum_{j=0}^m a_j \frac{(s_j)_r}{(n)_r} = \begin{cases}
        0,\quad &0\le r < m,\\
        (-1)^{m} 2^{1-2m}\frac{n^m (n-m)!}{ n!},\quad &r=m.
    \end{cases}
\end{equation}
At $r=m$, the factor ${(-1)^m2^{1-2m}}$ is the reciprocal of the leading coefficient of $T_m(1-2t)$. The other factors are from~\eqref{eq: W coefficient}. More importantly, 
\begin{equation*}
    \mu_{m,n}(\{\bzero\}) = a_0,\qquad \|\mu_{m,n}\|_{\rm TV} = \sum_{j=0}^m |a_j|.
\end{equation*}
Therefore, using $|F(t)|\le 1$ on $t\in[0, 1]$ and~\eqref{eq: cond}, we only need to show 
\begin{equation}\label{eq: new cond}
a_0 \ge 0,\qquad  \sum_{j=0}^m |a_j| - 2 a_0 \le 1.
\end{equation}
\subsection{Interpolation}\label{sec: inter}
The equality constraints~\eqref{eq: mu cons} can be viewed as polynomial interpolation at distinct nodes $s_0,\dots, s_m$:
\begin{equation*}
    p_r(s) =  \sum_{j=0}^m a_j \frac{(s)_r}{(n)_r}\in\bbP_r.
\end{equation*}
Thus, $\{p_r\}_{r=0}^m$ forms a basis of $\bbP_m$, and the quadrature rule $(a_j, s_j)_{j=0}^m$ essentially extracts the coefficient of $s^m$ with a rescaling $\frac{(-1)^m n^m}{2^{2m-1}}$. To find the weights $a_j$, we consider the polynomials 
\begin{equation*}
    q_j(s) = \prod_{l=0, l\neq j}^{m} \left(\frac{2s_l}{n} - \frac{2s}{n}\right), \qquad 0\le j\le m,
\end{equation*}
then 
\begin{equation*}
    \sum_{l=0}^m a_l q_j(s_l) = \frac{(-1)^m n^m}{2^{2m-1}} \cdot \frac{(-1)^m}{n^m} 2^m = \frac{1}{2^{m-1}}. 
\end{equation*}
The only surviving term on the left-hand side is $a_jq_j(s_j)$, thus 
\begin{equation}\label{eq: aj lag}
   a_j = 2^{1-m} \left( \prod_{l=0, l\neq j} \left(\frac{2s_l}{n} - \frac{2s_j}{n}\right) \right)^{-1}.
\end{equation}
In particular, for the Chebyshev polynomial $T_m(1 - 2s/n)$, the leading coefficient is $2^{2m-1}(-n)^{-m}$. Therefore, 
\begin{equation*}
    \sum_{j=0}^m a_j T_m\left(1-\frac{2s_j}{n}\right) = 1.
\end{equation*}
This relation immediately gives an estimate of the total variation of $\mu_{m,n}$:
\begin{equation*}
    1 =\left|\sum_{j=0}^m a_j T_m\left(1-\frac{2s_j}{n}\right) \right| \le \|\mu_{m,n}\|_{\rm TV} \|T_m\|_{\infty} = \|\mu_{m,n}\|_{\rm TV},
\end{equation*}
which leads to a lower bound $ \|\mu_{m,n}\|_{\rm TV} \ge 1 $.
When the above equality holds, we need the interpolation points to match the Chebyshev extrema:
\begin{equation}\label{eq: cheby}
    1 - \frac{2s_j^{\ast}}{n} = \cos\left(\frac{j\pi}{m}\right).
\end{equation}
Then, the weight $a_j^{\ast}$ follows~\eqref{eq: aj lag}:
\begin{equation}\label{eq: aj}
    a_j^{\ast} = 2^{1-m}\left( \prod_{l=0, l\neq j}^m \left(\cos(j\pi/m) - \cos(l\pi/m) \right) \right)^{-1} =  \begin{cases}
    (-1)^j \frac{1}{2m}, &\quad  j = 0, m,\\
     (-1)^j \frac{1}{m}, &\quad  1\le j \le m-1.
    \end{cases}
\end{equation}
The proof of the identities~\eqref{eq: aj} is in Appendix~\ref{sec: aj}.
This shows that
\begin{equation}
    a_0^{\ast} = \frac{1}{2m} > 0,\qquad \sum_{j=0}^m |a_j^{\ast}| - 2 a_0^{\ast} \le\left(1 - \frac{2}{2m}\right) <  1.
\end{equation}
Thus the desired condition~\eqref{eq: new cond} holds with the above $s_j^{\ast}$ and $a_j^{\ast}$, but the Chebyshev extrema~\eqref{eq: cheby} \textit{cannot} force all $s_j^{\ast}$ to be integers. A natural choice is to set $s_j$ to the closest integer to $s_j^{\ast}$ in~\eqref{eq: cheby}, {which shifts $s_j^{\ast}$ by at most $1/2$.} We still need to show that such a small perturbation won't destroy the condition~\eqref{eq: new cond}. 

\begin{remark}
    This relaxed lower-bound problem is closely related to finding approximate Chebyshev nodes (extrema are nodes of second kind) from an equispaced grid. These nodes are known as ``mock-Chebyshev''. It is a well-studied topic in numerical analysis~\cite{coppersmith1992growth,platte2011impossibility,adcock2019optimal}. The first gap between Chebyshev extrema is $1 - \cos(\pi/m) \sim \frac{\pi^2}{2m^2}$, which needs to be resolved by the grid size $\frac{2}{n}$ for a stable interpolation. This gives a conjectured optimal constant $C_{\ast} = 4/\pi^2$. The stability of interpolation is precisely characterized by $\sum_{i=0}^m |a_i|$.
\end{remark}

\subsection{Perturbative estimate}
In this section, we estimate the effect of the tiny perturbations from the Chebyshev extrema.  Take 
\begin{equation}\label{eq: tj}
   c_j = \cos(j\pi/m),\qquad  s_j = \left\lfloor \frac{1}{2} + \frac{n(1 - c_j)}{2}\right\rfloor,\qquad t_j = 1 - \frac{2s_j}{n},\qquad 0\le j\le m. 
\end{equation}
Then $s_j$ is the integer that minimizes $|c_j - t_j|$, and $s_0 = 0$, $s_m = n$, $t_0 = c_0 = 1$, $t_m = c_m = -1$. It remains to estimate the correponding $\{a_j\}_{j=0}^m$ for the new nodes $\{t_j\}_{j=0}^m$.

\begin{lemma}\label{lem: sep}
If $n > m^2$, then $0 = s_0 < s_1 < \dots < s_m$.
\end{lemma}
\begin{proof}
    Clearly, the sequence is non-decreasing. If $s_j = s_{j+1}$ for some $0\le j\le m-1$, then
    \begin{equation*}
        \left| \frac{n(1 - c_j)}{2} - \frac{n(1 - c_{j+1})}{2} \right| \le 1 \Longrightarrow |c_j - c_{j+1}|\le \frac{2}{n}.
    \end{equation*}
    On the other hand, 
    \begin{equation*}
    |c_{j} - c_{j+1}| = 2 \sin\left(\frac{\pi}{2m}\right)\sin\left(\frac{2j+1}{2m}\pi\right) \ge 2 \left(\frac{2}{\pi}\frac{\pi}{2m}\right)^2 = 2m^{-2} > \frac{2}{n}.
    \end{equation*}
    Contradiction.
\end{proof}
We can now prove a slightly stronger lower bound, which proves Theorem~\ref{thm: low}. 
\begin{lemma}\label{thm: low strong}
Let $n\ge 4m^2$. The matrix $\bW_{m,n}$ in~\eqref{eq: W} satisfies 
\begin{equation*}
    \bW_{m, n} \succeq -\left(1 - \frac{1}{2m}\right)\mathbf{I}.
\end{equation*}
\end{lemma}

\begin{proof}
We first estimate $a_j$ from the shifted nodes $\{t_l\}_{l=0}^m$ in~\eqref{eq: tj}. By Lemma~\ref{lem: sep}, $\{t_l\}_{l=0}^m$ are distinct, and~\eqref{eq: aj lag} gives 
\begin{equation*}
\begin{aligned}
a_j &= 2^{1-m} \left( \prod_{l=0, l\neq j} \left(\frac{2s_l}{n} - \frac{2s_j}{n}\right) \right)^{-1} = 2^{1-m} \left( \prod_{l=0, l\neq j} \left(t_j - t_l\right) \right)^{-1} \\
&=2^{1-m} \left( \prod_{l=0, l\neq j} \left(c_j - c_l\right) \right)^{-1} \left(\prod_{l=0, l\neq j} \frac{t_j - t_l}{c_j - c_l}\right)^{-1}\\
&= a_j^{\ast} \left( \prod_{l=0, l\neq j} \frac{t_j - t_l}{c_j - c_l} \right)^{-1},
\end{aligned}
\end{equation*}
where $a_j^{\ast}$ is the weight in~\eqref{eq: aj}. Because $\{t_j\}_{j=0}^m$ and $\{c_j\}_{j=0}^m$ are both strictly decreasing sequences, ${\rm sgn}(a_j) = {\rm sgn}(a_j^{\ast}) = (-1)^j$. Since $j\neq l$, $|j-l|\le j+l\le 2m-|j-l|$ and $1\le |j-l|\le m$, we have
\begin{equation}\label{eq: c-c}
\begin{aligned}
     \left|\frac{t_j - t_l}{c_j - c_l}\right| &= \left|1 + \frac{(t_j - c_j) - (t_l - c_l)}{{c_j - c_l}}\right| \ge 1 - \frac{2}{n|c_j - c_l|} \\&= 1 - \frac{1}{n\sin(\frac{(j+l)\pi}{2m})\sin(\frac{|j-l|\pi}{2m})}\\
     &\ge 1 - \frac{1}{n \left( \frac{2}{\pi} \cdot \frac{|j-l|\pi}{2m} \right) \left(\frac{2}{\pi} \frac{|j-l|\pi}{2m}\right)} = 1 - \frac{m^2}{n|j-l|^2}.
\end{aligned}
\end{equation}
Since $n\ge 4m^2$, $\frac{m^2}{n|j-l|^2} \le \frac{1}{4}$, using $1 - x \ge e^{-2x}$ for $0\le x\le \frac{1}{4}$, we obtain
\begin{equation*}
     \left|\frac{t_j - t_l}{c_j - c_l}\right| \ge \exp\left(- \frac{2m^2}{n|j-l|^2}\right).
\end{equation*}
Therefore,
\begin{equation}\label{eq: aj up}
    |a_j|\le |a_j^{\ast}| \prod_{l=0, l\neq j} \exp\left(\frac{2m^2}{n |j-l|^2}\right) \le |a_j^{\ast}| \exp\left(\frac{2m^2}{n}\frac{\pi^2}{3}\right),
\end{equation}
where we have used $\sum_{l\ge 1} l^{-2} = \pi^2/6$ counted in two directions. For $a_0$, it is positive,  we use $c_0 = t_0=1$ and follow a similar derivation in~\eqref{eq: c-c}, 
\begin{equation*}
\begin{aligned}
    a_0 = a_0^{\ast} \left(\prod_{l\ge 1}^m \frac{1 - t_l}{1 - c_l} \right)^{-1} &\ge a_0^{\ast} \prod_{l\ge 1}^m\left(1 + \frac{1}{n(1 - c_l)}\right)^{-1} \ge a_0^{\ast} \prod_{l\ge 1}^m \left(1 + \frac{m^2}{2 n l^2}\right)^{-1}\\
    &\ge a_0^{\ast} \exp\left(-\frac{\pi^2}{12}\frac{m^2}{n}\right) = \frac{1}{2m} \exp\left(-\frac{\pi^2}{12}\frac{m^2}{n}\right) \ge \frac{1}{2m}\exp\left(-\frac{\pi^2}{48}\right) > \frac{2}{5m}, 
\end{aligned}
\end{equation*}
where we use $1 + x \le e^x$ for $x\ge 0$. Now, we estimate $\sum_{j=0}^m |a_j|$. Apply $(a_j, s_j)_{j=0}^m$ to Chebyshev polynomial $T_m(1 - 2s/n)$ and note that $T_m(c_j) = {(-1)^j}$,
\begin{equation*}
\begin{aligned}
    1 &= \sum_{j=0}^m a_j T_m(1 - 2 s_j/n) =  \sum_{j=0}^m a_j T_m(t_j) \\
    &=\sum_{j=0}^m a_j (T_m(t_j) - {(-1)^j}) + \sum_{j=0}^m a_j {(-1)^j} \\
    &= \sum_{j=0}^m a_j (T_m(t_j) - {(-1)^j}) + \sum_{j=0}^m |a_j|.
\end{aligned}
\end{equation*}
The first term can be bounded by
\begin{equation}\label{eq: abs}
\begin{aligned}
    \left| \sum_{j=0}^m a_j (T_m(t_j) - {(-1)^j}) \right| &\le \sum_{j=0}^m |a_j| |T_m(t_j) - T_m(c_j) |.
\end{aligned}
\end{equation}
The difference $T_m(t_j) - T_m(c_j) = 0$ for $j =0, m$. For $0 < j < m$, let $\theta_j^{\ast} = \arccos(c_j) = j\pi/m$ and $\theta_j = \arccos(t_j)$, then $\sin^2\theta_j^{\ast} \ge \sin^2(\pi/m) \ge \frac{4}{m^2} > \frac{4}{n}$. By $\arccos'(t) = -\frac{1}{\sqrt{1-t^2}}$, we have 
\begin{equation*}
\begin{aligned}
|\theta_j^{\ast} - \theta_j|\le \frac{1}{n}\sup_{u \text{ between } {c_j, t_j}} \frac{1}{\sqrt{1-u^2}} &\le \frac{1}{n}\sup_{u \text{ between } {c_j, t_j}}  \frac{1}{\sqrt{1-c_j^2 - |c_j^2 - u^2|}}  \\
&\le \frac{1}{n}\sup_{u \text{ between } {c_j, t_j}}  \frac{1}{\sqrt{\sin^2(\theta^{\ast}_j) - \frac{2}{n}}}\\
&\le \frac{1}{n} \frac{\sqrt{2}}{\sin\theta_j^{\ast}}.
\end{aligned}
\end{equation*}
Then $m|\theta_j^{\ast} - \theta_j| \le \frac{m\sqrt{2}}{n \sin\theta_j^{\ast}} \le \frac{m^2\sqrt{2}}{2n} \le \frac{\sqrt{2}}{8} < \frac{\pi}{2}$, we have 
\begin{equation*}
\begin{aligned}
|T_m(t_j) - T_m(c_j) | &= 1 - \cos(m(\theta_j^{\ast} - \theta_j)) \\&\le \frac{m^2|\theta_j^{\ast} - \theta_j|^2}{2}\le \frac{m^2}{n^2 \sin^2\theta_j^{\ast}} \le \frac{m^4}{4 n^2} \cdot \frac{1}{\min(j^2, (m-j)^2) }.
\end{aligned}
\end{equation*}
Plugging into the right-hand side of~\eqref{eq: abs}, and apply~\eqref{eq: aj up} and~\eqref{eq: aj},
\begin{equation*}
\begin{aligned}
    \sum_{j=1}^{m-1} |a_j| |T_m(t_j) - T_m(c_j) | &\le  \frac{1}{m} \exp\left(\frac{2m^2}{n}\frac{\pi^2}{3}\right) \frac{m^4}{4 n^2} \sum_{j=1}^{m-1}  \frac{1}{\min(j^2, (m-j)^2) } \\
    &\le \exp\left(\frac{\pi^2}{6}\right) \frac{\pi^2}{12} \frac{m^3}{n^2} \le \exp\left(\frac{\pi^2}{6}\right) \frac{\pi^2}{192} \frac{1}{m} \\&< \frac{3}{10 m}.
\end{aligned}
\end{equation*}
Combine the above estimates, 
\begin{equation*}
   {\sum_{j=0}^m |a_j| - 2 a_0} \le 1 + \frac{3}{10m} - \frac{4}{5m}= 1 - \frac{1}{2m}.
\end{equation*}
\end{proof}

\section{Proof on upper bound}\label{sec: high}
Using the same measure $\mu_{m,n}$ in~\eqref{eq: mu cons}, the following is an identity
\begin{equation}\label{eq: W id}
    \bW_{m,n} = \int_{\{0, 1\}^n} T_m(\mathbf{I}-2\bB(\bz)) d\mu_{m,n}(\bz).
\end{equation}
To prove $\mathbf{I} \succeq \bW_{m,n}$, we need to decompose $\mathbf{I} - \bW_{m,n}$ into PSD parts. This is essentially equivalent to the Positivstellensatz argument. The measure $\mu_{m,n}(\bz)$ makes this more explicit, 
\begin{equation}\label{eq: W id2}
      \bI - \bW_{m,n} = \bI - \sum_{j=0}^m a_j \int_{\{0, 1\}^n }T_m\left(\mathbf{I} - 2 \bB(\bz) \right) d\nu_{s_j}(\bz).
\end{equation}
Let us emphasize how this differs from the lower bound.
In the previous proof, we could use the positivity of $\mu_{m,n}(\{\bzero\})$ to relax the inequality with perturbative estimates because the inequality itself is not sharp. More importantly, we also do not require PSD for each summand, and a crude bound $\|T_m\left(\mathbf{I} - 2 \bB(\bz) \right)\|_{\rm op}\le 1$ suffices. Here, running the same technique requires stronger arguments: 
\begin{enumerate}
    \item proving $\|\mu_{n,m}\|_{\rm TV}=1$, which is only possible with Chebyshev extrema, and $s_j$ \textit{may not} be integers. 
    \item finding a PSD decomposition of~\eqref{eq: W id2} on the right-hand side.
\end{enumerate}
Next, we address these two issues.
\subsection{Decomposition}
A key observation is that~\eqref{eq: W id2} actually does \textit{not} require $s_j$ to be integers. The sacrifice is that non-integer $s$ does not yield a non-negative vertex measure $\nu_{s}$, but we can understand it as analytic interpolation. Nonnegativity on integers does not imply non-negativity elsewhere. The Runge phenomenon shows that near the endpoints, the non-negativity can be fragile. Thus, the difficulty lies when $s_j$ is close to $0$ and $n$. Repeating the calculations in~\eqref{eq: chby} and~\eqref{eq: fb} without the integer assumption, we have the following extension for non-integer $\{s_j\}_{j=0}^m$.
\begin{lemma}
    Suppose there are $m+1$ distinct real numbers $\{ s_j\}_{j=0}^m$ that
    \begin{equation}\label{eq: aj con2}
    \sum_{j=0}^m a_j \frac{(s_j)_r}{(n)_r} = \begin{cases}
        0,\quad &0\le r < m,\\
        (-1)^{m} \frac{n^m (n-m)!}{2^{2m-1} n!},\quad &r=m.
    \end{cases}
    \end{equation}
    Then~\eqref{eq: W id2} holds, with $\nu_s$ defined as a linear mapping from $\{0, 1\}^n$ to $\bbR$:
    \begin{equation}\nonumber
    \int_{\{0, 1\}^n} z_{j_1}\cdots z_{j_r} d\nu_s(\bz) = \frac{(s)_r}{(n)_r} .
\end{equation}
\end{lemma} 
Then we can choose the same quadrature rule for Chebyshev extrema as in Section~\ref{sec: inter}:
\begin{equation}\nonumber
\begin{aligned}
s_j = \frac{n}{2} \left(1 - \cos\left(\frac{j\pi}{m}\right)\right) = n\sin^2 \frac{j\pi}{2m},\qquad a_j = (-1)^j\eps_j,\qquad\eps_j = \begin{cases}
    \frac{1}{2m},\qquad j=0, m,\\
    \frac{1}{m},\qquad 1\le j\le m-1.
\end{cases}
\end{aligned}
\end{equation}
Then using $\sum_{j=0}^m \eps_j = 1$, we take a parity decomposition:
\begin{equation*}
\begin{aligned}
    \mathbf{I} - \mathbf{W}_{n,m} &= \sum_{j=0}^m \eps_j \left(\mathbf{I} - (-1)^j\int_{\{0, 1\}^n }T_m\left(\mathbf{I} - 2 \bB(\bz) \right) d\nu_{s_j}(\bz)\right) \\
    &= \sum_{j=0,\; j\text{ even}}^m \eps_j \int_{\{0, 1\}^n } \left(\mathbf{I} - T_m\left(\mathbf{I} - 2 \bB(\bz) \right)\right) d\nu_{s_j}(\bz) \\
    &\quad + \sum_{j=0,\; j\text{ odd}}^m \eps_j \int_{\{0, 1\}^n } \left(\mathbf{I} + T_m\left(\mathbf{I} - 2 \bB(\bz) \right)\right) d\nu_{s_j}(\bz) .
\end{aligned}
\end{equation*}
For any fixed $\bz\in\{0,1\}^n$, both $\mathbf{I} \pm T_m\left(\mathbf{I} - 2 \bB(\bz)\right)$ are PSD, and each permits an explicit decomposition into a sum of squares using the other kinds of Chebyshev polynomials.
\begin{lemma}\label{lem: ss}
Let $\bX = \mathbf{I} - 2\bB(\bz)$. 
There exists matrix-coefficient polynomials $\bF_{+}(\bz), \bF_{-}(\bz)$ of degree $\le \lceil \frac{m}{2} \rceil$ such that 
\begin{equation*}
\begin{aligned}
    \mathbf{I} + T_{m}(\bX) &= \bF_{m,+}(\bz) \bF_{m,+}(\bz)^{\top},\\\mathbf{I} - T_{m}(\bX) &= \bF_{m,-}(\bz) \bF_{m,-}(\bz)^{\top}.
\end{aligned}
\end{equation*}
\end{lemma}
\begin{proof}
Let $U_r$, $V_r$, and $W_r$ denote the degree-$r$ Chebyshev polynomials of the second, third, and fourth kind, respectively. Then
\begin{equation*}
    U_r(\cos\theta) = \frac{\sin((r+1)\theta)}{\sin\theta}, \qquad V_r(\cos\theta) = {\frac{\cos((r+1/2)\theta)}{\cos(\theta/2)}},\qquad W_r(\cos\theta) = {\frac{\sin((r+1/2)\theta)}{\sin(\theta/2)}}.
\end{equation*}
If $m=2r$, we have the following identities
\begin{equation*}
\begin{aligned}
    1 + \cos(2r\theta) &=2\cos^2(r\theta),\\
    1 - \cos(2r\theta) &= 8\left( U_{r-1}(\cos\theta) \right)^2 \left( \left(\frac{1 - \cos\theta}{2}\right) - \left(\frac{1 - \cos\theta}{2}\right)^2 \right), 
\end{aligned}
\end{equation*}
Use $\bB = \bB(\bz) = \frac{\bI -\bX}{2}$, they imply
\begin{equation*}
\begin{aligned}
    \mathbf{I} + T_{2r}(\bX) &= 2 T_r(\bX)^2 \Longrightarrow \bF_{2r, +} = \sqrt{2} T_r(\bX). \\
    \mathbf{I} - T_{2r}(\bX) &= 8U_{r-1}(\bX) \left(\bB - \bB^2\right) U_{r-1}(\bX). 
\end{aligned}
\end{equation*}
Next, we find the square root of $\bB - \bB^2$.
Let $\bR = (\sqrt{\bE_1},  \dots, \sqrt{\bE_n})$, $\bZ = \bZ(\bz) = {\rm diag}(z_i \mathbf{I})_{i=1}^n$ with $z_i\in\{0, 1\}$, then $\bZ^2 = \bZ$, and 
\begin{equation*}
    \bB = \bR \bZ \bZ \bR^{\top},\qquad \bB^2 = \bR \bZ \bR^{\top} \bR \bZ \bR^{\top}.
\end{equation*}
Therefore, let $\mathbf{I}_n$ denote the identity matrix in dimension $n\dim(\bX)$, we have  
\begin{equation*}
    \bB - \bB^2 = \bR\bZ \bQ \bQ^{\top} \bZ \bR^{\top},\qquad \bQ = \left(\mathbf{I}_n - \bR^{\top} \bR\right)^{1/2}  .
\end{equation*}
Thus, $\bF_{2r, -} = 2\sqrt{2} U_{r-1}(\bX) \bR\bZ \bQ$. 
In both cases, the degrees of $\bz$ are at most $r$.
If $m=2r+1$, 
\begin{equation*}
\begin{aligned}
    1 + \cos((2r+1)\theta) &=  V_r^2(\cos\theta)(1 + \cos\theta) \\
    1 - \cos((2r+1)\theta) &= W_r^2(\cos\theta) (1 - \cos\theta)
\end{aligned}
\end{equation*}
Therefore, using $z_i^2 = z_i$ and $(1-z_i)^2 = 1-z_i$, 
\begin{equation*}
\begin{aligned}
     \mathbf{I} - T_{2r+1}(\bX) &= 2 \sum_{i=1}^n W_r(\bX) z_i \sqrt{\bE_i} {\left(W_r(\bX) z_i \sqrt{\bE_i}\right)^{\top}} \\
      \mathbf{I} + T_{2r+1}(\bX) &= 2\sum_{i=1}^n V_r(\bX) (1-z_i) \sqrt{\bE_i} \left(V_r(\bX) (1-z_i) \sqrt{\bE_i} \right)^{\top} + {2V_r(\bX) \bH \bH^{\top} V_r(\bX)},
\end{aligned}
\end{equation*}
where $\bH\bH^{\top} = \bI - \sum_{i=1}^n \bE_i$. Therefore, 
\begin{equation*}
\begin{aligned}
     {\bF_{2r+1,-}} &= (\sqrt{2} W_r(\bX) z_1 \sqrt{\bE_1}, \dots, \sqrt{2} W_r(\bX) z_n \sqrt{\bE_n} ), \\
     {\bF_{2r+1,+}} &= (\sqrt{2} V_r(\bX) (1-z_1) \sqrt{\bE_1}, \dots, \sqrt{2}V_r(\bX) (1-z_n) \sqrt{\bE_n}, {\sqrt{2}V_r(\bX)\bH}), 
\end{aligned}
\end{equation*}
All components have degrees in $\bz$ that are at most $r+1$.
\end{proof}

\subsection{Non-negativity on squares}
Take any fixed nonzero vector $\bv$, by Lemma~\ref{lem: ss}, 
\begin{equation}\label{eq: control}
\begin{aligned}
    \aver{\bv, \bI - {\bW_{m,n}}\bv} &=
     \sum_{j=0,\; j\text{ odd}}^m \eps_j \int_{\{0, 1\}^n } \|\bF_{m, +}^{\top}(\bz) \bv\|^2 d\nu_{s_j}(\bz) \\
    &\quad +\sum_{j=0,\; j\text{ even}}^m \eps_j \int_{\{0, 1\}^n } \| \bF_{m, -}^{\top}(\bz)\bv\|^2 d\nu_{s_j}(\bz) .
\end{aligned}
\end{equation}
Since each entry of $\bF_{m, \pm}^{\top}(\bz) \bv$ is a polynomial of $\bz$ with degree no more than $\lceil \frac{m}{2} \rceil$, the above can be formulated as 
\begin{equation}\label{eq: sum f}
\begin{aligned}
    \aver{\bv, \bI - {\bW_{m,n}}\bv} &= \sum_{j=0,\; j \text{ odd}}^m  \eps_j \sum_{\alpha}\int_{\{0,1\}^n} |f_{\alpha}(\bz)|^2 d\nu_{s_j}(\bz),\\
    &\quad + \sum_{j=0,\; j \text{ even}}^m  \eps_j \sum_{\beta}\int_{\{0,1\}^n} |f_{\beta}(\bz)|^2 d\nu_{s_j}(\bz),
\end{aligned}
\end{equation}
where $f_{\alpha}, f_{\beta}\in\bbP_{\lceil m/2\rceil} $. To show that $ \bI - {\bW_{m,n}} $ is PSD, it suffices to show that the measure $\nu_{s_j}$ maps polynomial squares to non-negative values for all $0\le j\le m$. The following theorem is equivalent to~\cite[Lemma 1.4]{grigoriev2001complexity}.
\begin{theorem}\label{thm: pos}
    Let $n\ge 2d$, and $s$ be any real number $d-1 \le s \le n-(d-1)$. Then, for any of degree-$d$ polynomial $f(\bz)$, we have
    \begin{equation*}
        \int_{\{0, 1\}^n} |f(\bz)|^2 d\nu_s \ge 0.
    \end{equation*}
\end{theorem}
This theorem shows that the non-negativity interpolates when $s$ is far from the two endpoints $0$ and $n$. The main idea of its proof is from the following Lemma~\ref{lem: har} on harmonic decomposition. See~\cite [Theorem 6.1]{filmus2019harmonicity} for a reference. We include proofs of Theorem~\ref{thm: pos} and Lemma~\ref{lem: har} in Appendix~\ref{sec: pos} for completeness. 
\begin{lemma}\label{lem: har}
    Let $2d\le n$. Then any degree-$d$ polynomial $f(\bz)$ can find another degree-$d$ polynomial $\widetilde{f}(\bz)$ such that 
    \begin{equation*}
        \int_{ \{0, 1\}^n  } |f(\bz)|^2 d\nu_s(\bz) = \int_{ \{0, 1\}^n  } |\widetilde{f}(\bz)|^2 d\nu_s(\bz),
    \end{equation*}
    and $\widetilde{f}$ consists of Boolean harmonic functions
    \begin{equation*}
        \widetilde{f}(\bz) = \sum_{j=0}^d h_j,\qquad h_j = \sum_{|S|=j} c_S \prod_{i\in S} z_i, \qquad \Delta h_j = 0,
    \end{equation*}
    where the operator $\Delta = \sum_{i=1}^n \partial_{z_i}$. The coefficient $c_S$ is a polynomial of $s$ with degree $\le d - |S|$. 
    Furthermore, 
    \begin{equation}\label{eq: coef}
   \int_{ \{0, 1\}^n  } |\widetilde{f}(\bz)|^2 d\nu_s(\bz) = \sum_{j=0}^d \frac{(s)_j (n-s)_j}{(n)_{2j}}\left(\sum_{|S|=j} |c_S|^2 \right),
\end{equation}
The coefficient $\frac{(s)_j (n-s)_j}{(n)_{2j}}$ is non-negative if $d-1\le s \le n-(d-1)$.
\end{lemma}
Using the result above, we can prove the following theorem when all non-integer values in $\{s_j\}_{j=0}^m$ satisfy the condition outlined in Theorem~\ref{thm: pos}. 
\begin{theorem}\label{thm: m3}
    Let $d = \lceil \frac{m}{2} \rceil$, and $n \ge 2d$, and 
    \begin{equation*}
        n \sin^2 \frac{\pi}{2m} \ge d- 1.
    \end{equation*}
    Then $\|\bW_{m,n}\|_{\rm op}\le 1$, that is, the AM-GM inequality~\eqref{eq: rr} holds on both bounds.
\end{theorem}
\begin{proof}
     {The case $m=1$ is the scalar arithmetic-mean bound. For $m\ge2$, set $s_j=n\sin^2(j\pi/(2m))$.} We find that $s_0 = 0$, $s_m = n$ are integers, thus $\nu_{s_0}$ and $\nu_{s_m}$ are non-negative automatically. For $1\le j\le m-1$, 
     \begin{equation*}
         s_1=n\sin^2\frac{\pi}{2m} < s_2 < \cdots < s_{m-1} = n\sin^2\frac{(m-1)\pi}{2m} = n - n\sin^2\frac{\pi}{2m}.
     \end{equation*}
     The given condition implies that 
     \begin{equation*}
         d-1\le s_1 < \cdots < s_{m-1} \le n -(d-1).
     \end{equation*}
     Applying Theorem~\ref{thm: pos} to the square factors of both $\bI-T_m$ and $\bI+T_m$ gives $\aver{\bv,(\bI\pm\bW_{m,n})\bv}\ge0$ for every vector $\bv$. Thus $\|\bW_{m,n}\|_{\rm op} \le 1$.
\end{proof}
\begin{corollary}
    Suppose $n\ge m^3/2$ and $n\ge m$. The AM-GM inequality~\eqref{eq: rr} holds on both sides.
\end{corollary}
\begin{proof}
If $m=1$, the result follows directly from $0\preceq n^{-1}\sum_i\bA_i\preceq\bI$. For $m\ge2$, $n\ge m^3/2$ also implies $n\ge2d$, and $n\sin^2\frac{\pi}{2m} \ge \frac{n}{m^2}\ge \frac{m}{2} \ge d-1$. Then apply the previous theorem.
\end{proof}
This corollary requires $m = \cO(n^{1/3})$ for the AM-GM inequality to hold. The proof essentially constructs the sum of squares (SOS) required by the Positivstellensatz argument. The coefficients in~\eqref{eq: coef} first turn negative at $s=d-2$ and $s=n-d+2$, and their signs alternate as $s$ moves toward the two ends. To strengthen the result to $n = \Omega(m^2)$, an intuitive idea is to carefully control the negative contribution from~\eqref{eq: control} so that the positive contribution can absorb it. However, there is a beautiful way to avoid this.

\subsection{Parity quadrature rules}
By~\eqref{eq: sum f}, we only have to show 
\begin{equation*}
    \begin{aligned}
         \sum_{j=0,\; j \text{ odd}}^m  \eps_j \sum_{\alpha}\int_{\{0,1\}^n} |f_{\alpha}(\bz)|^2 d\nu_{s_j}(\bz) &\ge 0,\\
       \sum_{j=0,\; j \text{ even}}^m \eps_j \sum_{\beta}\int_{\{0,1\}^n} |f_{\beta}(\bz)|^2 d\nu_{s_j}(\bz) &\ge 0.
    \end{aligned}
\end{equation*}
The parity polynomials $f_{\alpha}$ and $f_{\beta}$ are generally different, so we cannot combine them. Using Lemma~\ref{lem: har}, 
each degree-$d$ polynomial $f_{\alpha}$ has another degree-$d$ polynomial $\widetilde{f}_{\alpha}$ (same for $f_{\beta}$): 
\begin{equation*}
\begin{aligned}
  \int_{\{0,1\}^n} |f_{\alpha}(\bz)|^2 d\nu_{s_j}(\bz) &=   \int_{\{0,1\}^n} |\widetilde{f}_{\alpha}(\bz)|^2 d\nu_{s_j}(\bz)\\
  &= \sum_{k=0}^d \frac{(s_j)_k (n-s_j)_k}{(n)_{2k}}\left(\sum_{|S|=k} |c_{\alpha, S}(s_j)|^2 \right),
\end{aligned}
\end{equation*}
where each $c_{\alpha, S}$ is a polynomial in $s$ of degree $\le d-|S|$. Thus, after an affine map $s_j = \frac{n}{2}(1 - t_j)$, $t_j = \cos(\frac{j\pi}{m})$, 
it suffices to show that for any given $0\le k\le d$,
\begin{equation}\label{eq: tar}
    \begin{aligned}
         \sum_{j=0,\; j \text{ odd}}^m  \eps_j \frac{(\frac{n}{2}(1 - t_j))_k (\frac{n}{2}(1 + t_j))_k}{(n)_{2k}} |p(t_j)|^2 &\ge 0,     \qquad \forall p\in\bbP_{d - k}.\\
       \sum_{j=0,\; j \text{ even}}^m \eps_j \frac{(\frac{n}{2}(1 - t_j))_k (\frac{n}{2}(1 + t_j))_k}{(n)_{2k}} |p(t_j)|^2  &\ge 0,    \qquad \forall p\in\bbP_{d - k}.
    \end{aligned}
\end{equation}
When $k=0$, nothing needs to be proven. From now on, we assume $1\le k\le d$ is fixed.

Note that $\sum_{j\text{ odd}}\eps_j = \sum_{j\text{ even}}\eps_j = \frac{1}{2}$, and both summands in~\eqref{eq: tar} are degree-$2d$ in $t$, which motivates us to define the parity quadrature rules for any $f\in \bbP_{2d}$:
\begin{equation*}
    \cQ_{-1}(f) :=  2 \sum_{j=0,\; j \text{ odd}}^m  \eps_j f(t_j) ,\qquad \cQ_{+1}(f) :=  2 \sum_{j=0,\; j \text{ even}}^m  \eps_j f(t_j) .
\end{equation*}
Let $\omega = e^{i\theta}$, and $t = \frac{\omega+\omega^{-1}}{2}$. Any polynomial $f(t)\in\bbP_{2d}$ can be written as its Laurent series
\begin{equation}\label{eq: G}
  G(\omega) =  f\left(\frac{\omega + \omega^{-1}}{2}\right) = \sum_{j=-2d}^{2d} g_{j} \omega^j.
\end{equation}
Thus, we can redefine the two quadrature rules into a unified form for Laurent series for $0 < |\omega| < \infty$,
\begin{equation*}
    \cQ_{\eta}(G) := \frac{1}{m} \sum_{\omega^m = \eta}  G(\omega),           \qquad\eta \in \{-1, 1\} .
\end{equation*}
By the residue theorem, the quadrature rule is a contour integral on the boundary of the annulus $\Gamma$ (positively oriented) containing the unit circle: 
\begin{equation}\label{eq: cont}
     \cQ_{\eta}(G) = \frac{1}{2\pi i} \int_{\Gamma} \frac{G(\omega) \omega^{m-1}}{\omega^m - \eta} d\omega, \qquad \eta \in \{-1, 1\}.
\end{equation}
\begin{lemma}\label{lem: Q eta}
    Let $d = \lceil \frac{m}{2} \rceil$ with $m\ge 2$, then 
    \begin{equation*}
        \cQ_{\eta}(G) = g_0 + \eta (g_m + g_{-m}), \qquad \eta \in \{-1, 1\}.
    \end{equation*}
    where $g_m, g_{-m}$ are the coefficients in~\eqref{eq: G}.
\end{lemma}
\begin{proof}
    Take the inner boundary $\Gamma_{-}$ (negative oriented) and outer boundary $\Gamma_{+}$ (positive oriented) of the annulus. On $\Gamma_{+}$, 
    \begin{equation*}
      \frac{G(\omega) \omega^{m-1}}{\omega^m - \eta} = \frac{1}{\omega} \left( \sum_{l=0}^{\infty} \left( \frac{\eta}{\omega^m} \right)^l \right) \sum_{j=-2d}^{2d} g_j \omega^j
    \end{equation*}
    The contour integral on $\Gamma_{+}$ is nonzero only at $j \in  m\bbN$, since $2d < 2m$ for $m\ge 2$, leaving only $j = 0$ and $j = m$. Thus, 
    \begin{equation*}
        \frac{1}{2\pi i} \int_{\Gamma_{+}}     \frac{G(\omega) \omega^{m-1}}{\omega^m - \eta} d\omega = g_0 + \eta g_m.
    \end{equation*}
    Similarly, on $\Gamma_{-}$, 
        \begin{equation*}
      \frac{G(\omega) \omega^{m-1}}{\omega^m - \eta} = -\frac{1}{\omega} \frac{\omega^{m}}{ \eta} \left( \sum_{l=0}^{\infty} \left( \frac{\omega^m}{\eta} \right)^l \right) \sum_{j=-2d}^{2d} g_j \omega^j.
    \end{equation*}
    The contour integral is nonzero only when $-(m+j)\in m\bbN$, so $j = -m$ is the only choice. Thus,  
    \begin{equation*}
        \frac{1}{2\pi i} \int_{\Gamma_{-}}  \frac{G(\omega) \omega^{m-1}}{\omega^m - \eta} d\omega = \eta g_{-m}.
    \end{equation*}
\end{proof}
Next, we make the target in~\eqref{eq: tar} more explicit. 
\begin{lemma}\label{lem: fall}
The following identity holds
$$\frac{(\frac{n}{2}(1 - t))_k (\frac{n}{2}(1 + t))_k}{(n)_{2k}} =  \frac{\left(\frac{n^2}{4}\right)^k}{(n)_{2k}} \prod_{l=0}^{k-1}  \left( 1 - t^2 - \beta_l \right),\qquad \beta_l:= \frac{4l(n-l)}{n^2},\qquad 0\le k\le d.$$
When $n\ge 2d$, the constant $\frac{\left(\frac{n^2}{4}\right)^k}{(n)_{2k}}$ is non-negative. 
\end{lemma}
\begin{proof}
A direct computation gives the identity.
    \begin{equation*}
\begin{aligned}
    \frac{(\frac{n}{2}(1 - t))_k (\frac{n}{2}(1 + t))_k}{(n)_{2k}}  &= \frac{1}{(n)_{2k}}\prod_{l=0}^{k-1} \left(\frac{n}{2}(1 - t) - l\right) \left(\frac{n}{2}(1 + t) - l\right)  \\&= \frac{1}{(n)_{2k}}\prod_{l=0}^{k-1} \frac{n^2}{4} \left( 1 - t^2 - \frac{4l(n-l)}{n^2} \right)\\
    &= \frac{\left(\frac{n^2}{4}\right)^k}{(n)_{2k}}\prod_{l=0}^{k-1}  \left( 1 - t^2 - \beta_l \right),\qquad \beta_l:= \frac{4l(n-l)}{n^2}.
\end{aligned}
\end{equation*}
\end{proof}
Taking $t = \frac{\omega + \omega^{-1}}{2}$ gives $1 - t^2 =  -\frac{(\omega^2 - 1)^2}{4\omega^2}$.  Define
\begin{equation*}
     \Phi_k(\omega) := \prod_{l=0}^{k-1}  \left( u(\omega) - \beta_l \right), \qquad  u(\omega) := -\frac{(\omega^2 - 1)^2}{4\omega^2}.
\end{equation*}
It can be written as 
\begin{equation}\label{eq: phi k}
    \Phi_k(\omega) = u^k + \sum_{l=1}^{k-1} (-1)^l \cS_l u^{k-l},
\end{equation}
The term $\beta_0 = 0$, and the coefficient $\cS_l$ is the symmetric sum of $(\beta_1,\cdots, \beta_{k-1})$ of length $l$:
\begin{equation}\label{eq: sym}
    \cS_l:= \cS_l(\beta_1,\cdots, \beta_{k-1}) = \sum_{\substack{|S|=l\\S\subset \{1,\dots, k-1\}}} \prod_{j\in S} \beta_{j}.
\end{equation}

\subsection{Expansion of target polynomial}
Let $L$ be a Laurent series
\begin{equation*}
    L(\omega) = \sum_{j\in\bbZ} c_j \omega^j, 
\end{equation*}
and we define its paraconjugate by
\begin{equation*}
     L^{\sharp}(\omega) = \sum_{j\in\bbZ} \overline{c_j} \omega^{-j}. 
\end{equation*}
 In particular, 
\begin{equation}\label{eq: ush}
    u(\omega) = \frac{- (\omega^2 - 1)^2}{4\omega^2} = \frac{1}{4} (\omega^2 - 1) (\omega^{-2} - 1) = \frac{1}{4} (\omega^2 - 1) (\omega^{2} - 1)^{\sharp}. 
\end{equation}
Given any $p\in\bbP_{d-k}$, let $$Y(\omega) =  p\left(\frac{\omega + \omega^{-1}}{2}\right) = \sum_{l= -(d-k)}^{d-k} y_l \omega^l, $$
and  $Y^{\sharp}(\omega) = \overline{Y(\omega)}$ on the unit circle.
Therefore, using Lemma~\ref{lem: fall}, the desired~\eqref{eq: tar} is equivalent to proving
\begin{equation*}
    \cQ_{\eta} (  \Phi_{k}(\omega) Y(\omega) Y^{\sharp}(\omega) ) \ge 0.
\end{equation*}
By~\eqref{eq: phi k}, it permits the following expansion
\begin{equation}\label{eq: exp}
\cQ_{\eta} (  \Phi_{k}(\omega) Y(\omega) Y^{\sharp}(\omega) ) = \cH_k(Y) + \sum_{l=1}^{k-1} (-1)^l \cS_l  \cH_{k-l}(Y), 
\end{equation}
where $\cH_{j}(Y) := \cQ_{\eta} (u^{j}(\omega) Y(\omega) Y^{\sharp}(\omega) )$, $1 \le j \le k$. We aim to show that $\cH_k(Y)$ dominates the remaining terms. To find each $\cH_{j}(Y)$, we introduce
\begin{equation}\label{eq: pj}
    P_j(\omega) := (\omega^2 - 1)^j \omega^{d-k} Y(\omega).
\end{equation}
Then $P_j$ is a polynomial, and $  u^j Y Y^{\sharp} = 4^{-j} P_j  P_j^{\sharp} $.

Define the coefficient norm for polynomials by 
\begin{equation*}
    \|q\|_{\rm coeff}^2 := \sum_{k=0}^{\deg(q)} |c_k|^2, \qquad\text{if}\quad q(\omega) = \sum_{k=0}^{\deg(q)} c_k \omega^k.
\end{equation*}
This norm induces a Hilbert space $\mathbb{H}_{\rm coeff}$ with inner product
\begin{equation*}
    \aver{p, q}_{\rm coeff} = \frac{1}{2\pi} \int_{0}^{2\pi} p^{\sharp}(e^{i\theta}) q(e^{i\theta}) d\theta.
\end{equation*}
{On polynomial coefficient vectors, the adjoint of multiplication by $1-\omega$ is multiplication by $(1-\omega)^{\sharp}$ followed by removing negative powers.}  The following lemma characterizes $\cH_j$ in the coefficient norm.
\begin{lemma}\label{lem: rem}
Let $r_j\in\bbP_{m-1}$ be the remainder of $P_j$ divided by $\omega^m - \eta$, 
\begin{equation*}
 r_j = P_j \mod (\omega^m - \eta).
\end{equation*}
Then
\begin{equation*}
    \cH_{j}(Y) = 4^{-j} \|r_j\|^2_{\rm coeff}.
\end{equation*}
\end{lemma}
\begin{proof}
   From the above analysis, we have 
    $\cH_{j}(Y) = 4^{-j} \cQ_{\eta} (P_j  P_j^{\sharp} )$. Since $\omega^m - \eta = 0$ on each quadrature node, $P_j = r_j$ on these nodes. Therefore, $ \cQ_{\eta} (P_j  P_j^{\sharp} ) =  \cQ_{\eta} (r_j  r_j^{\sharp} )$. Because $\deg(r_j)\le m-1$, the exponents in $r_j r_j^{\sharp}$ range from $-(m-1)$ to $(m-1)$. By Lemma~\ref{lem: Q eta}, 
    \begin{equation*}
         \cQ_{\eta} (r_j  r_j^{\sharp} ) = [\omega^0] r_j  r_j^{\sharp} = \|r_j\|_{\rm coeff}^2. 
    \end{equation*}
\end{proof}
We emphasize that ${\deg(P_j)\le 2d+2j-2k}$.
If $1\le j\le k - 1$, this degree is less than $m$, so the modulo is not needed, $\|r_j\|_{\rm coeff}^2 = \|P_j\|_{\rm coeff}^2$. The modulo is needed only when $j=k$.

\subsection{Iterative estimates}
In this part, we estimate $\cH_{j}(Y)$. We decompose the polynomials into their even- and odd-degree terms.
\begin{equation*}
    P_{j}(\omega) = A_j(\omega^2) + \omega B_j(\omega^2),\qquad {\deg(A_j)\le d-k+j,\qquad \deg(B_j)\le d-k+j-1.}
\end{equation*}
Then by the definition~\eqref{eq: pj}, $P_{j+1}(\omega) = P_{j}(\omega) (\omega^2 - 1)$, we have the iteration
\begin{equation}\label{eq: iter}
    A_{j+1}(u) = A_j (u)(u-1),\qquad B_{j+1}(u) = B_j (u)( u - 1), \qquad u = \omega^2,
\end{equation}
and coefficient norm $\|P_j\|^2_{\rm coeff} = \|A_j\|^2_{\rm coeff} + \|B_j\|^2_{\rm coeff}$. 
\begin{lemma}\label{lem: iter est}
   Let $1\le j \le L-1$. Suppose that 
    \begin{equation*}
\begin{aligned}
    &\varphi \in V:= \operatorname{span}\{ (u-1)^j, (u-1)^{j+1}, \cdots, (u-1)^{L-1}\}.
\end{aligned}
\end{equation*}
Then 
\begin{equation*}
    j(j+1) \|\varphi\|_{\rm coeff}^2 \le \frac{L^2}{4} \|(u-1)\varphi\|_{\rm coeff, 1}^2,
\end{equation*}
where $\|\cdot\|_{\rm coeff, \tau}$ is a semi-norm by removing the first and last $\tau$ terms:
\begin{equation*}
    \|q\|_{\rm coeff, \tau}^2 := \sum_{k=\tau}^{\deg(q)-\tau} |c_k|^2, \qquad\text{if}\quad q(\omega) = \sum_{k=0}^{\deg(q)} c_k \omega^k.
\end{equation*}
\end{lemma}
\begin{proof}
Let $\cD = u \frac{d}{du}$, then it is a self-adjoint operator on $\mathbb{H}_{\rm coeff}$ since $\cD u^l = l u^l$. We consider the operator $$\cH p :=  (u -1)^{\sharp}\cD (L - \cD) (u -1) p, $$ it satisfies
\begin{equation}\label{eq: tri}
\begin{aligned}
    \cH (u-1)^{l} &= (u^{-1}-1) \left( (L-1) u (l+1) (u-1)^{l} - (l+1)l u^2 (u-1)^{l-1} \right) \\
    &= -(l+1)(u-1)^{l} ((L-l-1) (u-1) - l)\\
    &= l(l+1)(u-1)^{l} - (l+1)(L-l-1) (u-1)^{l+1},
\end{aligned}
\end{equation} 
when $l = L-1$, the last term vanishes. Thus, $\cH$ is self-adjoint on $\mathbb{H}_{\rm coeff}$ and preserves $V_j$. The calculation~\eqref{eq: tri} indicates the representation under $V_j$ is triangular. The eigenvalues restricted on $V_j$ are the diagonal values $l(l+1)$, $j\le l \le L-1$. Therefore,
\begin{equation*}
    \aver{\varphi, \cH \varphi}_{\rm coeff} \ge j(j+1) \|\varphi\|_{\rm coeff}^2.  
\end{equation*}
Moreover, 
\begin{equation*}
\begin{aligned}
    \aver{\varphi, \cH \varphi}_{\rm coeff}  &= \aver{\varphi, (u -1)^{\sharp}\cD (L- \cD) (u -1)  \varphi}_{\rm coeff} \\
    &= \aver{(u -1) \varphi, \cD (L- \cD) (u -1)  \varphi}_{\rm coeff}.
\end{aligned}
\end{equation*}
The eigenvalues of $ \cD (L - \cD)$ are $l(L - l)\le \frac{L^2}{4}$ for $0\le l\le L$ with eigenvectors as the monomial $\{ u^l \}_{l=0}^L$. In particular, the monomials $u^0$ and $u^L$ are suppressed completely. Therefore, 
\begin{equation*}
j(j+1) \|\varphi\|_{\rm coeff}^2 \le \frac{L^2}{4} \|(u-1)\varphi\|_{\rm coeff, 1}^2 .
\end{equation*}
\end{proof}
Apply Lemma~\ref{lem: iter est} to both $A_j$ and $B_j$,
\begin{equation*}
    \begin{aligned}
A_j &\in \operatorname{span}\{ (u-1)^j, (u-1)^{j+1}, \cdots, (u-1)^{d-k+j}\}, \\
B_j &\in \operatorname{span}\{ (u-1)^j, (u-1)^{j+1}, \cdots, (u-1)^{d-k+j-1}\}.  
\end{aligned}
\end{equation*}
Then,
\begin{equation*}
\begin{aligned}
 \|A_j\|_{\mathrm{coeff}}^2
   &\le \frac{(d-k+j+1)^2}{4j(j+1)}
       \|A_{j+1}\|_{\mathrm{coeff},1}^2,\\
 \|B_j\|_{\mathrm{coeff}}^2
   &\le \frac{(d-k+j)^2}{4j(j+1)}
       \|B_{j+1}\|_{\mathrm{coeff},1}^2.
\end{aligned}
\end{equation*}
Combining them gives
\begin{equation}\label{eq: next}
4^{-j} \|P_j\|_{\rm coeff}^2 \le \frac{(d-k+j+1)^2}{j (j+1) }  \cdot 4^{-(j+1)}  {\|P_{j+1}\|_{\mathrm{coeff},2}^2},
\end{equation}
Here $\|\cdot\|_{\rm coeff, 2}$ is used since each parity drops two terms.
Repeating the above analysis iteratively, we obtain for $j=1,\dots, k-1$,
\begin{equation*}
\begin{aligned}
 \cH_{k-j}(Y) = {4^{-(k-j)}} \|P_{k-j}\|_{\rm coeff}^2 &\le \prod_{l=k-j}^{k-1} \frac{(d-k+l+1)^2}{l (l+1)}    4^{-k} {\|P_k\|_{\mathrm{coeff},2;2d}^2} \\
 &=  \frac{((d)_{j})^2}{(k)_{j}(k-1)_{j}}    4^{-k} {\|P_k\|_{\mathrm{coeff},2;2d}^2}.
\end{aligned}
\end{equation*}
We still need to find $\cH_k(Y)$. Let  $  P_k(\omega) = \sum_{l=0}^{2d} c_l \omega^l $. By Lemma~\ref{lem: rem}, if $2d = m$,
\begin{equation*}
   r_k =  P_k \mod (\omega^m - \eta) = (c_0 + \eta c_{2d}) + \sum_{l=1}^{2d-1} c_l \omega^l. 
\end{equation*}
If $2d = m+1$, 
\begin{equation*}
    r_k =  P_k \mod (\omega^m - \eta) = {(c_0+\eta c_{2d-1})+(c_1+\eta c_{2d})\omega} + \sum_{l=2}^{2d-2} c_l \omega^l. 
\end{equation*}
{In both cases, the modified terms are outside the indices $2,\ldots,2d-2$.} Therefore, 
\begin{equation*}
    4^{-k} {\|P_k\|_{\mathrm{coeff},2}^2} \le   4^{-k} \|r_k\|_{\rm coeff}^2 = \cH_k(Y).
\end{equation*}
The above analysis gives us key estimates:
\begin{equation}\label{eq: est 1}
\begin{aligned}
 \cH_{k-j}(Y) &\le \frac{((d)_{j})^2}{(k)_{j}(k-1)_{j}}  \cH_k(Y),\qquad 1\le j\le k-1,\\
  \cH_{k-j}(Y) &\le \frac{(d-j+1)^2}{(k-j)(k-j+1)} \cH_{k-j+1}(Y), \qquad 1\le j\le k-1.
\end{aligned}
\end{equation}

\subsection{Proof of Theorem~\ref{thm: both}}
{Assume $2d^2/n\le1$. The case $k=0$ is the non-negative average
$\cQ_\eta(YY^\sharp)$. For $k=1$, $\Phi_1=u\ge0$ on the unit circle,
so $\cQ_\eta(uYY^\sharp)\ge0$ directly. We may therefore assume $2\le k\le d$.} 
Recall the symmetric sum in~\eqref{eq: sym}:
\begin{equation*}
   \cS_l:= \cS_l(\beta_1,\cdots, \beta_{k-1}) = \sum_{\substack{|S|=l\\S\subset \{1,\dots, k-1\}}} \prod_{j\in S} \beta_{j}.
\end{equation*}
It satisfies a simple bound 
\begin{equation*}
       (l+1) \cS_{l+1} = \sum_{|S|=l} \prod_{j\in S} \beta_{j} \left(\sum_{r\notin S} \beta_r \right) \le  \cS_l  \sup_{|S|=l}  \left(\sum_{r\notin S} \beta_r \right). 
\end{equation*}
Since $ n \ge 2d \ge 2k$, $\beta_r = \frac{4r(n-r)}{n^2}\le \frac{4(k-1)(n-k+1) }{n^2}$ for all $1\le r \le k-1$. Then,  
\begin{equation*}
    \sum_{r\notin S} \beta_r \le (k-1-l) \frac{4(k-1)(n-k+1) }{n^2} \le \frac{4(k-1) (k-1-l)}{n}.
\end{equation*}
Thus,
\begin{equation}\label{eq: sym iter}
    \cS_{l+1} \le  \frac{4(k-1) (k-1-l)}{ (l+1) n} \cS_l.
\end{equation}
The first symmetric sum $\cS_1$ can be evaluated analytically,
\begin{equation}\label{eq: s1}
\begin{aligned}
  \cS_1 = \sum_{l=1}^{k-1}\beta_l &= \sum_{l=1}^{k-1} \frac{4 l (n-l)}{n^2} = \frac{2k(k-1)}{n} - \frac{4 }{n^2} \frac{(k-1) k (2k-1)}{6}\\
  &= \frac{2k(k-1) }{n} \left(  1- \frac{2k-1}{3n} \right) .
\end{aligned}
\end{equation}
Now combine the estimates~\eqref{eq: est 1}, ~\eqref{eq: sym iter} into~\eqref{eq: exp}. For $ 1\le l\le k-2$, 
\begin{equation*}
\begin{aligned}
     \cS_{l+1}  \cH_{k-l-1}(Y) &\le \frac{4(k-1) (k-1-l)}{ (l+1) n}  \frac{(d-l)^2}{(k-l-1)(k-l)} \cS_l  \cH_{k-l}(Y) \\
     &= \frac{4(k-1)(d-l)^2}{(l+1) n (k-l)} \cS_l  \cH_{k-l}(Y)\\
     &\le \frac{2 (d-1)^2}{n} \frac{2 (k-1) }{(l+1)(k-l)}\cS_l  \cH_{k-l}(Y)\\
     &\le \frac{2 (d-1)^2}{n} \cS_l  \cH_{k-l}(Y)  \le  \cS_l  \cH_{k-l}(Y).
\end{aligned}
\end{equation*}
This is because $\frac{2d^2}{n}\le 1$ and $k\ge l+2$ leads to  
\begin{equation*}
\begin{aligned}
    (l+1)(k-l) - 2(k-1) &= (l-1) k - l(l+1) + 2 \\&\ge (l-1) (l+2)  - l(l+1) + 2 = 0.
\end{aligned}
\end{equation*}
Recall~\eqref{eq: exp}, our target $\cQ_{\eta} (  \Phi_{k}(\omega) Y(\omega) Y^{\sharp}(\omega) )$ equals
\begin{equation*}
\begin{aligned}
&\cH_k(Y) + \sum_{l=1}^{k-1} (-1)^l \cS_l  \cH_{k-l}(Y) \\
&\quad= \cH_k(Y) - \cS_1 \cH_{k-1}(Y) + \sum_{l=1}^{{\lfloor (k-2)/2\rfloor} }\left( \cS_{2l}  \cH_{k-2l}(Y) -  \cS_{2l+1}  \cH_{k-2l-1}(Y)  \right) + \cR\\
&\quad \ge \cH_k(Y) - \cS_1 \cH_{k-1}(Y).
\end{aligned}
\end{equation*}
where $\cR = \cS_{k-1}  \cH_{1}(Y) \ge 0$ for odd $k$, and $\cR = 0$ for even $k$. Finally, we apply~\eqref{eq: est 1} and~\eqref{eq: s1}. Then
\begin{equation*}
\begin{aligned}
    \cH_k(Y) - \cS_1 \cH_{k-1}(Y) &\ge \cH_k(Y)\left( 1 - \frac{d^2}{k(k-1)} \frac{2k(k-1) }{n} \left(  1- \frac{2k-1}{3n} \right)  \right)\\
    &\ge \cH_k(Y) \left( 1 - \frac{2d^2}{n}  \left(  1- \frac{2k-1}{3n} \right) \right) \ge 0.
\end{aligned}
\end{equation*}
The last inequality holds since $\frac{2d^2}{n} \le 1$ and
\begin{equation*}
    3n  - (2k-1) \ge 6k^2 - 2k + 1 > 0.
\end{equation*}
The proof is complete.

\section{Concluding remarks}\label{sec: con}
This work proves the non-commutative AM-GM inequality for $n\ge 2\lceil m/2 \rceil^2$. This is the regime where sampling without replacement outperforms sampling with replacement. The core idea is to construct a vertex measure $\mu_{m,n} = {\sum_{j=0}^{m}a_j\nu_{s_j}}$ and integrate it with a Chebyshev polynomial. The main difficulty is that $\nu_{s_j}$ may be negative if $s_j$ is non-integer. The lower bound of~\eqref{eq: rr} is not sharp. Thus, we can sacrifice sharpness for flexibility to perturb $s_j$ to the nearest integers. This relates closely to mock-Chebyshev nodes in interpolation theory when the equispaced grid size $2/n$ must resolve the first Chebyshev gap of degree $m$. The upper bound is sharp. The key step is Lemma~\ref{lem: har}, which lets the surrogate polynomial reduce the problem back to a quadrature problem. 

\section*{Acknowledgements}
YZ was partially funded by NSF grant DMS-2309530.

\appendix
\section{\texorpdfstring{Proof of~\eqref{eq: eig}}{Proof of the simplex eigenvalue formula}}\label{sec: p eig}
Let $\bW_{m,n}$ {be constructed} in Section~\ref{sec: mot}. For $\bA_{i_k}\cdots\bA_{i_1}$ with distinct indices, we let 
\begin{equation*}
    c_k = \aver{\bu, \bA_{i_k}\cdots\bA_{i_1} \bu}.
\end{equation*}
Then $c_0 = c_1 = 1$. The eigenvalue $\lambda = c_m$. Note that $c_k$ does not depend on the choice of distinct indices. Applying the matrices $\bA_{i_j}$ to $\bu$ from right to left, by induction we find that $\bA_{i_l}\cdots\bA_{i_1} \bu$ is in the following form:
\begin{equation*}
    \bA_{i_l}\cdots\bA_{i_1} \bu = c_l \bu + \sum_{i=1}^l g_{i,l} \bz_i,
\end{equation*}
where the constants $g_{i,l}$ are not important individually, since all $\bz_i\perp \bu$ and $\bz_i^{\top} \bz_j = -1/d$ if $i\neq j$. We only need to record $\gamma_l = \sum_{i=1}^l g_{i,l}$. Therefore, after applying $\bA_{i_{l+1}}$ to $\bu$, $\bz_{i_j}$, $j=1,\dots, l$, 
\begin{equation*}
    \bA_{i_{l+1}} \bu = \bu + \beta \bz_{i_{l+1}} ,\qquad \bA_{i_{l+1}} \bz_{i_j} = \alpha \bz_{i_j} - \frac{\beta}{d} \bu - \frac{\beta^2 - \alpha}{d} \bz_{i_{l+1}}. 
\end{equation*}
Thus, 
\begin{equation*}
    \begin{pmatrix}
        c_{l+1}\\ \gamma_{l+1}
    \end{pmatrix} = \begin{pmatrix}
        c_l - \frac{\beta}{d} \gamma_l \\
        \beta c_l + (\alpha - {\frac{\beta^2-\alpha}{d}})\gamma_l
    \end{pmatrix} = \begin{pmatrix}
        1  & -\beta/d\\
        \beta & \alpha - (\beta^2 - \alpha)/d
    \end{pmatrix}      \begin{pmatrix}
        c_{l}\\ \gamma_{l}
    \end{pmatrix}.
\end{equation*}
It corresponds to the iteration $c_{k+2} = 2 \alpha c_{k+1} - \frac{\alpha n}{n-1} c_{k}$. For the specific choice~\eqref{eq: choice}, it becomes
\begin{equation}
    c_{k+2} = 2 \left(\frac{n}{n-1} \cos\theta\right) \cos\theta c_{k+1} - \left(\frac{n}{n-1} \cos\theta\right)^2 c_k.
\end{equation}
Thus, solving the iteration gives
\begin{equation*}
    c_{k} = \left(\frac{n}{n-1} \cos\theta\right)^{k-1} \frac{\sin(k\theta) - \left(\frac{n}{n-1} \cos\theta\right)\sin((k-1)\theta) }{\sin\theta}.
\end{equation*}
At $k = m$, $\sin((k-1)\theta) = 0$, and $\sin( k\theta) =(-1)^{\ell} \sin(\theta)$.
\section{\texorpdfstring{Proof of~\eqref{eq: aj}}{Proof of the interpolation weight formula}}\label{sec: aj}
Let $U_{k}$ be the degree-$k$ Chebyshev polynomial of the second kind. Then $$U_{m-1}(\cos\theta) = \sin(m\theta)/\sin\theta,$$ the extrema $\cos(j\pi/m)$, $j=1,\cdots, m-1$, are zeros of ${U_{m-1}=T_m'/m}$:
\begin{equation*}
    \prod_{j=0}^{m} \left( t - \cos(j\pi/m) \right) \equiv {2^{1-m}}(t^2-1)U_{m-1}(t).
\end{equation*}
The factor $2^{1-m}$ normalizes the leading coefficient. Therefore, if ${1\le j\le m-1}$ the formula~\eqref{eq: aj} gives
\begin{equation*}
\begin{aligned}
   (a_j^{\ast})^{-1} &=\frac{d}{dt}(t^2-1)U_{m-1}(t)\Big|_{t = \cos(j\pi/m)} \\&=  (t^2-1) U'_{m-1}(t)\Big|_{t = \cos(j\pi/m)}\\
   &= (m T_m(t) - t U_{m-1}(t)) \Big|_{t = \cos(j\pi/m)}\\
   &=  m (-1)^j.
\end{aligned}
\end{equation*}
For $j=0, m$, 
\begin{equation*}
    (a_j^{\ast})^{-1} =  \frac{t^2-1}{t - \cos(j\pi/m)} U_{m-1}(t) \Big|_{t = \cos(j\pi/m)} = \begin{cases}
         2 m , &\quad j = 0,\\
        2 m (-1)^{m}, &\quad j=m.
    \end{cases}
\end{equation*}

\section{Proof of Theorem~\ref{thm: pos} and Lemma~\ref{lem: har}}\label{sec: pos}

\begin{lemma}\label{lem: c1}
    For any $s$, let $Z = \sum_{i=1}^n z_i$, with $z_i\in\{0, 1\}$, then 
    \begin{equation}
        \begin{aligned}
            \int_{\{0, 1\}^n} (Z - s) f(\bz) d\nu_s(\bz) &= 0,\qquad &&\forall f\in\bbP_{n-1}.\\
            \int_{\{0,1\}^n} \prod_{j\in S} z_j \prod_{i\in T}(1-z_i) d\nu_s(\bz) & = \frac{(s)_{|S|} (n-s)_{|T|}}{(n)_{|S|+|T|}},\qquad &&\forall S\cap T =\emptyset. 
        \end{aligned}
    \end{equation}
\end{lemma}

\begin{proof}
    Consider a monomial $z_{j_1}\cdots z_{j_r}$, where $j_1, \dots, j_r$ distinct. Then,
    \begin{equation*}
    \begin{aligned}
             \int_{\{0, 1\}^n} Z z_{j_1}\cdots z_{j_r} d\nu_s &= r\frac{(s)_r}{(n)_r}  + (n-r) \frac{(s)_{r+1}}{(n)_{r+1}} \\
             &= r\frac{s(s-1)\cdots (s-r+1)}{n(n-1)\cdots (n-r+1)}  + (s-r) \frac{s(s-1)\cdots (s-r+1)}{n(n-1)\cdots (n-r+1)}\\
             &= s \frac{(s)_r}{(n)_r} =  s\int_{\{0, 1\}^n} z_{j_1}\cdots z_{j_r} d\nu_s,
    \end{aligned}
    \end{equation*}
    The first decomposition holds because exactly $n-r$ indices from $Z$ increase the degree from $r$ to $r+1$. The remaining indices keep the degree.  

    For all integer $s = 0,\dots, n$. Both sides of the second equation give the probability that a subset of size $s$ contains $S$ and excludes $T$. Treating both sides as polynomials in $s$, their degrees are at most ${|S|+|T|\le n}$. {The $n+1$ integer nodes $0,\ldots,n$ therefore determine these polynomials uniquely.}
\end{proof}
Now we prove Lemma~\ref{lem: har}. Lemma~\ref{lem: c1} shows that $Z - s$ is orthogonal to all polynomials in $\bbP_{n-1}$. Therefore, to construct another polynomial $\widetilde{f}$, we only need 
\begin{equation*}
    f^2(\bz) - \widetilde{f}^2(\bz) = (Z - s) g(\bz),\qquad \deg(g)\le n-1.
\end{equation*}
    Let $\cV_r$ be the space spanned by the basis monomial $\prod_{i\in S} z_i$ with $|S|=r$. The Laplace operator $\Delta = {\sum_{i=1}^{n}\partial_{z_i}}$ is a linear transform $\Delta: \cV_r\to \cV_{r-1}$ for $r\ge 1$ and 
    \begin{equation*}
        \Delta \prod_{i\in S} z_i = \sum_{i\in S} \prod_{j\in S - \{i\}} z_j.
    \end{equation*}
    The adjoint operator $\Delta^{\ast}: \cV_r\to \cV_{r+1}$ can be defined from the intrinsic inner-product structure $L^2(\cV_r; \bbC)$. Let $p, q\in\cV_r$ be 
    \begin{equation*}
        p = \sum_{|S|=r} p_S \prod_{i\in S} z_i,\qquad   q = \sum_{|S|=r} q_S \prod_{i\in S} z_i.
    \end{equation*}
    Then $\aver{p, q}_{\cV_r} := \sum_{|S|=r} \overline{p_S} q_S $, and the $\Delta^{\ast}$ is 
    \begin{equation*}
        \Delta^{\ast} \prod_{i\in S} z_i = \sum_{j\notin S} \prod_{i\in S\cup \{j\} } z_i
    \end{equation*}
    Consider any $p\in (\Delta^{\ast} \cV_r)^{\perp} \subset \cV_{r+1}$, then for any $q\in\cV_r$, we have
    \begin{equation*}
{       0=\aver{p,\Delta^{\ast}q}=\aver{\Delta p,q}\quad\text{for all }q\in\cV_r
       \quad\Longleftrightarrow\quad p\in\ker(\Delta)\cap\cV_{r+1}.} 
    \end{equation*}
Therefore, $\cV_{r+1} = \Delta^{\ast} \cV_r \oplus \left( \ker(\Delta)\cap \cV_{r+1}\right) $. For any polynomial $f = \sum_{r=0}^d f_r$ of degree $d$, each $f_r$ is a homogeneous degree-$r$ polynomial. Then we apply the space decomposition repeatedly. Set $\theta_{r, r} = f_r$, and for $k=r, r-1,\dots, 1$, 
\begin{equation*}
    \theta_{r, k} = h_{r, k} + \Delta^{\ast} \theta_{r, k-1},\qquad h_{r, k}\in \ker(\Delta)\cap \cV_{k},   \qquad \theta_{r, k-1}\in \cV_{k-1}.
\end{equation*}
Here we do not require $\theta_{r, k}$ to be unique. At $k=0$, $h_{r, 0} = \theta_{r, 0}$. Sum these terms up, and we obtain
\begin{equation*}
    f_r = \sum_{k=0}^r ( \Delta^{\ast} )^{r - k} h_{r, k},\qquad h_{r, k}\in \ker(\Delta)\cap \cV_{k}.
\end{equation*}
Therefore, there exists a decomposition of expanded Boolean harmonic functions 
\begin{equation*}
    f = \sum_{r=0}^d \sum_{k=0}^r ( \Delta^{\ast} )^{r - k} h_{r, k}.
\end{equation*}
For any $p\in\cV_r$, we have
\begin{equation*}
    Z p = \Delta^{\ast} p + r p\Longrightarrow \Delta^{\ast} p = (Z - r)p,\qquad \forall p\in \cV_r.
\end{equation*}
Apply the above to $( \Delta^{\ast} )^{r - k} h_{r, k} $ repeatedly, 
\begin{equation*}
\begin{aligned}
    ( \Delta^{\ast} )^{r-k} h_{r, k} & = (Z - (k+r-k-1)) ( \Delta^{\ast} )^{r-k-1} h_{r, k} \\& = \cdots = (Z - r + 1)\cdots (Z - k) h_{r, k} \\
    &= (Z - k)_{r - k} h_{r, k}. 
\end{aligned}
\end{equation*}
Therefore, 
\begin{equation*}
    f = \sum_{r=0}^d \sum_{k=0}^r  (Z - k)_{r - k} h_{r, k}. 
\end{equation*}
We construct another polynomial $\widetilde{p}$ by replacing $Z$ with $s$, 
\begin{equation}\label{eq: w f}
    \widetilde{f} = \sum_{r=0}^d \sum_{k=0}^r  (s - k)_{r - k} h_{r, k}. 
\end{equation}
We emphasize that each term in $ \widetilde{f}$ is also a polynomial in $s$ and their total degree satisfies
\begin{equation}\label{eq: tot}
    \deg_s ((s - k)_{r - k} h_{r, k} ) + \deg_{\bz}((s - k)_{r - k} h_{r, k}) = r \le d.
\end{equation}
Then, $f - \widetilde{f} = (Z - s) g$ with $\deg(g)\le d-1$. Indeed, $\deg h_{r,k} = k$, 
\begin{equation*}
    f - \widetilde{f} = (Z - s) \sum_{r=0}^d \sum_{k=0}^r  J_{r, k}(Z) h_{r, k}, \qquad J_{r, k}(u) := \frac{  (u - k)_{r - k} - (s - k)_{r - k} }{u - s}\in\bbP_{r-k-1}.
\end{equation*}
Because $Z - s$ is real-valued polynomial, 
\begin{equation*}
    |f|^2 - |\widetilde{f}|^2 = (f - \widetilde{f}) \overline{\widetilde{f}} + \overline{(f - \widetilde{f})} f = ( Z -s ) ( \overline{g} f + g \overline{\widetilde{f}} ).
\end{equation*}
By Lemma~\ref{lem: c1}, 
\begin{equation*}
    \int_{ \{0, 1\}^n } \left(  |f|^2 - |\widetilde{f}|^2 \right) d\nu_s = 0.
\end{equation*}
Rearrange the terms in~\eqref{eq: w f} according to the degree in $\bz$, we can write 
\begin{equation*}
    \widetilde{f}(\bz) = \sum_{j=0}^d h_j,\qquad h_j = \sum_{|S|=j} c_S \prod_{i\in S} z_i, \qquad \Delta h_j = 0,
\end{equation*}
By~\eqref{eq: tot}, each $c_S$ is a polynomial of $s$ with degree $\le d - |S|$.
\begin{lemma}\label{lem: c3}
    If $h_j = \sum_{|S|=j} c_S \prod_{i\in S} z_i$ and $\Delta h_j = 0$, then given a fixed set $R$ that $|R| < j$, 
    \begin{equation*}
        \sum_{R\subset S} c_S = 0.
    \end{equation*}
\end{lemma}
\begin{proof}
    Let $|R| = l < j$, then in $\Delta^{j-l} h_j$, the coefficient of $\prod_{i\in R} z_i$ is $(j-l)! \sum_{R\subset S} c_S$. The factorial means there are $(j-l)!$ ways of deleting indices to reach $R$. Because $\Delta h_j = 0$, $\Delta^{j-l} h_j = 0$ whenever $l < j$. Hence $\sum_{R\subset S} c_S = 0$.
\end{proof}
Suppose $j < k$, 
\begin{equation*}
\begin{aligned}
    \int_{ \{0, 1\}^n } h_j\overline{h_k} d\nu_s &= \sum_{|S|=j}\sum_{|T|=k} c_S \overline{c_T} \int_{ \{0, 1\}^n } \prod_{i\in S\cup T} z_i d\nu_s  \\
    &= \sum_{|S|=j}\sum_{|T|=k} c_S \overline{c_T} \frac{(s)_{j+k - |S\cap T|}}{(n)_{j+k - |S\cap T|}} .
\end{aligned}
\end{equation*}
Here $s, n$ are fixed, let
\begin{equation}\label{eq: st}
    \varphi(|S\cap T|) := \frac{(s)_{j+k - |S\cap T|}}{(n)_{j+k - |S\cap T|}}. 
\end{equation}
Since $|S\cap T|$ only has possible values in $\{0, \dots, j\}$. There are constants $b_0, \dots, b_j$ that 
\begin{equation*}
    \varphi(|S\cap T|) = \sum_{l=0}^j b_l \binom{|S\cap T|}{l},
\end{equation*}
where $b_l$ depends on $s, n, j, k, l$ only. 
\begin{lemma}\label{lem: c4}
    The leading coefficient 
    \begin{equation*}
    b_j = \sum_{v=0}^j (-1)^{j-v} \binom{j}{v} \varphi(v).
\end{equation*}
\end{lemma}
\begin{proof}
Bring~\eqref{eq: st} into above,  
\begin{equation*}
 \sum_{v=0}^j (-1)^{j-v} \binom{j}{v}  \binom{v}{l} =  \sum_{v=l}^j (-1)^{j-v} \binom{j}{v}  \binom{v}{l} = \binom{j}{l} \sum_{v=l}^j (-1)^{j-v} \binom{j-l}{v - l} = \delta_{lj}.
\end{equation*}
The last identity is exactly the binomial theorem.
\end{proof}
Thus, by Lemma~\ref{lem: c3} and $l\le j < k$, 
\begin{equation*}
    \begin{aligned}
 \int_{ \{0, 1\}^n } h_j\overline{h_k} d\nu_s &= \sum_{l=0}^j b_l \sum_{|S|=j}\sum_{|T|=k} c_S \overline{c_T} \binom{|S\cap T|}{l} \\
&= \sum_{l=0}^j b_l \sum_{|R|=l} \left( \sum_{\substack{|S|=j\\ R\subset S} } c_S  \right)  \left( \sum_{\substack{|T|=k\\ R\subset T}} \overline{c_T}  \right) = 0.
    \end{aligned}
\end{equation*}
If $j = k$, in the above summation, only $R = S = T$ survives, 
\begin{equation*}
     \int_{ \{0, 1\}^n } |\widetilde{f}|^2 d\nu_s =  \sum_{j=0}^d \int_{ \{0, 1\}^n } |h_j|^2 d\nu_s = {\sum_{j=0}^d b_j\sum_{|S|=j}|c_S|^2}.
\end{equation*}
Let $S$, $T$ be two index sets that $|S|=|T|=j$, $S\cap T=\emptyset$, this is possible when $2j \le n$. Then by Lemma~\ref{lem: c4} and binomial expansion, 
\begin{equation*}
\begin{aligned}
    b_j &= \sum_{v=0}^j (-1)^{j-v}\binom{j}{v} \frac{(s)_{2j - v}}{(n)_{2j-v}} \\&= \int_{ \{0, 1\}^n }   \prod_{i\in S, |S|=j} z_i \prod_{k\in T, |T|=j}(1 - z_k) d\nu_s(\bz), \qquad 
\end{aligned} 
\end{equation*}
By Lemma~\ref{lem: c1}, $    b_j =  \frac{(s)_{j} (n-s)_{j}}{(n)_{2j}}$, where $0\le j\le d$. When $s \ge d-1$ and $s \le n-d+1$, both factors 
\begin{equation*}
    (s)_j = s(s-1) \cdots (s-j+1) \ge 0, \qquad (n-s)_{j} = (n-s)(n-s-1)\cdots (n- s - j + 1) \ge 0.
\end{equation*}
This proves both Theorem~\ref{thm: pos} and Lemma~\ref{lem: har}.

\bibliographystyle{siam}
\bibliography{main}
\end{document}